\documentclass{article}
\usepackage[utf8]{inputenc}

\title{Split-harmonic maps and the 
interpolation problem for timelike minimal surfaces}
\author{ Sreedev Manikoth\thanks{BS-MS Programme, Indian Institute of Science, Education and Research, Pune, Maharashtra, India, Email id:Sreedev.m@students.iiserpune.ac.in} }

\usepackage[utf8]{inputenc}
\usepackage{url}
\usepackage{hyperref}
\usepackage{mathtools}  
\usepackage{xfrac} 
\usepackage{amsmath}
\usepackage{parskip}
\usepackage[english]{babel}
\usepackage{amsthm}
\usepackage[english]{babel}
\usepackage{amsthm}
\usepackage{txfonts}
\newtheorem{theorem}{Theorem}[section]
\newtheorem{lem}[theorem]{Lemma}
\theoremstyle{definition}
\newtheorem{definition}{Definition}[section]

\theoremstyle{remark}

\begin{document}

\maketitle

\begin{abstract}
 The singular Bj\" orling problem and its solution for timelike minimal surfaces is a well-known result in minimal surface theory. In this article, we give a different proof of this theorem using split-harmonic maps. This is motivated by a similar solution of the singular Bj\"orling problem for maximal surfaces using harmonic maps. As an application, we study the problem of interpolating a given split-Fourier curve to a point by a timelike minimal surface. This is inspired by an analogous result for maximal surfaces. We also solve the problem of interpolating a given split-Fourier curve to another specified split-Fourier curve by a timelike minimal surface.
    
\end{abstract}

\section{Introduction}
\footnote{Mathematics Subject Classification. Primary 53A10, Key words and phrases. timelike minimal surfaces, split-holomorphic maps.}
In general zero-mean curvature surfaces are called minimal surfaces. The theory of minimal surfaces in Euclidean-3 space is rich with several beautiful theorems. One of the famous results in this subject is the Bj\"orling problem and its solution.

In Lorentz-Minkowski space, $\mathbb{L}^3$ zero-mean curvature surfaces can be spacelike or timelike. Spacelike zero mean curvature surfaces are called maximal surfaces, while timelike zero mean curvature surfaces are referred to as timelike minimal surfaces. Like minimal surfaces in Euclidean-3 space, one can use complex analysis to study maximal surfaces. But for timelike minimal surfaces, we have to use split-complex analysis.

Singular Bj\"orling problem for timelike minimal surfaces asks: given a lightlike curve $\gamma: (a,b) \rightarrow \mathbb{L}^3$ and a lightlike vector field $L :(a,b) \rightarrow \mathbb{L}^3 $ can we find a timelike minimal surface $X$ such that $X(u,0)=\gamma(u)$ and $X_v(u,0)=L(u)$. In \cite{Yw}, Y.W Kim, S.E Koh, and S-E Yang solve this using results involving the wave equation. We also note that in \cite{Rpr}, R.Dey, P.Kumar and R.K.Singh solves singular Bj\"orling problem for maximal surfaces when $\gamma: \mathbb{S}^1 \rightarrow \mathbb{L}^3$, $L: \mathbb{S}^1 \rightarrow \mathbb{L}^3$ are lightlike using a representation formula of maximal surfaces involving harmonic maps. This motivates us to solve the singular Bj\"orling problem for timelike minimal surfaces using a representation formula involving split-harmonic maps. We define split-harmonic maps as follows,

\begin{definition}
A map $f$: $\Omega \subset $ $\mathbb{C}^{'}$ $\rightarrow$ $\mathbb{R}$ is said to be split-harmonic if, $$f_{xx}-f_{yy}=0.$$

A map  $F$: $\Omega \subset $ $\mathbb{C}^{'}$ $\rightarrow$ $\mathbb{C}^{'}$ is said to be split-harmonic if each of its component functions is split-harmonic.
\end{definition}

Throughout this article, we will be using the following definition for $\mathbb{L}^3$.

\begin{definition}
$\mathbb{L}^3$ is $\mathbb{R}^3$ with the metric $ds^2=-dx^2+dy^2+dz^2$.
\end{definition}

We define split-Fourier curves.

We  note that for split-exponential map, $e^{k'\theta}=\cosh{\theta}+k^{'}\sinh{\theta}$. We refer to \cite{Fa}, page 11 for more details.

\begin{definition}
A curve $\gamma: \mathbb{H}^1 \rightarrow \mathbb{L}^3$ is said to be a split-Fourier curve if it has a finite series expansion of the form(I.e, only finitely many terms in the infinite series being nonzero), $\gamma(\theta)=(\gamma_1+k^{'}\gamma_2,\gamma_3)=(\Sigma_{-\infty}^{\infty} c_n e^{k^{'}n\theta}, \Sigma_{-\infty}^{\infty} d_n e^{k^{'}n\theta})$ with $c_n$, $d_n$ being split-complex numbers and  $\gamma_3(\theta)=\Sigma_{-\infty}^{\infty} d_n e^{k^{'}n\theta}$ being a real valued function.
\end{definition}

For example $$\gamma(\theta)= (\cosh \theta, \sinh \theta, \sinh \theta)= (\cosh \theta+ k^{'}\sinh \theta, \sinh \theta)=  (e^{k^{'}\theta}, \sinh \theta)= (e^{k^{'}\theta}, \frac{e^{k^{'}\theta}- e^{-k^{'}\theta}}{2 k^{'}})$$ is a split-Fourier curve.

In this paper, we solve singular Bj\"orling problem for timelike minimal surfaces when  $\gamma: \mathbb{H}^1 \rightarrow \mathbb{L}^3$, $L: \mathbb{H}^1 \rightarrow \mathbb{L}^3$ are lightlike. Here $\mathbb{H}^1$ denotes the set $\{x+k^{'}y \in \mathbb{C}^{'}|x>0,x^2-y^2=1\}$. $\mathbb{H}^1$ is the right branch of the unit hyperbola $x^2-y^2=1$. 

We first develop a representation formula for timelike minimal surfaces involving split-harmonic maps. Then we use this to solve the singular Bj\"orling problem. In pursuit of this, we show some results in split-complex analysis which are analogs of corresponding ones in complex analysis. As an application, we study the problem of interpolating a given spacelike or timelike split-Fourier curve to a point $p$  by a timelike minimal surface. This is analogous to a result in \cite{Rpr}.

Next, we take two arbitrary split-Fourier curves and find conditions such that there is a timelike minimal surface interpolating them.

\section{A representation formula for timelike minimal surfaces}

We recall the definition of generalized timelike minimal surfaces. Let $z=x+k^{'}y \in \Omega \subseteq \mathbb{C}^{'}$ . Here $k^{{'}^{2}}=1$ and $|z|^2=y^2-x^2$. We refer to \cite{Rb} for more about split-complex numbers. We recall the definitions of $\frac{\partial F}{\partial z}$ and $\frac{\partial F}{\partial \overline{z}}$ from \cite{Rb}.

$$\frac{\partial F}{\partial z}=\frac{1}{2}(\frac{\partial F}{\partial x}+k^{'}\frac{\partial F}{\partial y})$$ and

$$\frac{\partial F}{\partial \overline{z}}=\frac{1}{2}(\frac{\partial F}{\partial x}-k^{'}\frac{\partial F}{\partial y}).$$

The following definition is motivated by results in \cite{Rb}

\begin{definition}
A smooth map $F=(u,v,\omega):\Omega \rightarrow \mathbb{L}^3$ is said to be a generalized timelike minimal surface if it satisfies,
$$F_{xx}-F_{yy}=0,$$
$$\langle F_{x},F_{x} \rangle + \langle F_{y},F_{y} \rangle=0,$$ 
 $$\langle F_{x},F_{y} \rangle=0$$
and $$-\left|\frac{\partial u}{\partial z}\right|^2+\left|\frac{\partial v}{\partial z}\right|^2+\left|\frac{\partial \omega}{\partial z}\right|^2 \not\equiv 0 \hspace{0.1cm}on \hspace{0.1cm}\Omega.$$  
\end{definition}
We identify $\mathbb{L}^3$  with $\mathbb{C}^{'} \times \mathbb{R}$. Let,

$$h=u+k^{'}v.$$

Also, we define
$$\phi_1=\frac{\partial u}{\partial z}, \phi_2=\frac{\partial v}{\partial z} \hspace{0.1cm} and \hspace{0.1cm}
\phi_3=\frac{\partial \omega}{\partial z}.$$

Now we present the main result of this section.

\begin{theorem}
A smooth map $F=(h,w):\Omega \rightarrow \mathbb{L}^3$ is  a  generalized timelike minimal surface if and only if it is split-harmonic with $\omega_z^2= h_z \overline{h_{\overline{z}}} $ and $|h_z|$ not identically same as $|h_{\overline{z}}|$.
\end{theorem}

\begin{proof}

We prove in the forward direction. 

Since $$F_{xx}-F_{yy}=0, $$ $F$ is split-harmonic.
By computations we get,
$$4(-\phi_1^{2}+\phi_2^{2}+\phi_3^{2})=\langle F_x, F_x \rangle + \langle F_y, F_y \rangle + 2 k^{'} \langle F_x, F_y \rangle,$$ 

$$h_z=\phi_1+k^{'} \phi_2$$ and $$\overline{h_{\overline{z}}} =\phi_1-k^{'} \phi_2.$$ 

This implies,
$$4(\omega_z^2-h_z \overline{h_{\overline{z}}})= \langle F_x, F_x \rangle + \langle F_y, F_y \rangle + 2 k^{'} \langle F_x, F_y \rangle.$$

Since $F$ is represented by conformal parameters, $$\omega_z^2=h_z \overline{h_{\overline{z}}}.$$

Using this, one can show by calculating that $$-\left|\frac{\partial u}{\partial z}\right|^2+\left|\frac{\partial v}{\partial z}\right|^2+\left|\frac{\partial \omega}{\partial z}\right|^2= \frac{-(|h_z|-|\overline{h_{\overline{z}}}|)^2}{2}.$$

From here we get \hspace{0.001cm} $|h_z|$ not identically same as $|h_{\overline{z}}|$. Proof of the other direction is similar.
\end{proof}

In particular above representation formula shows that for time-like minimal graphs over $x-y$ plane, $h$ is injective. 

\section{On split-harmonic maps}

In this section, we state some results about split-harmonic maps. We start with the definition of split-holomorphic and split-analytic maps. We refer to \cite{Rb} for more details.

\begin{definition}[split-holomorphic]
A map $f=u+iv:\Omega \rightarrow \mathbb{C}^{'}$ is said to be split-holomorphic if for any $z=x+k^{'}y \in \Omega$, $u$ and $v$ satisfies,
$$u_x=v_y$$ 
$$u_y=v_x.$$
We call the above equations as Cauchy-Riemann equations in split-complex analysis.
\end{definition}

\begin{definition}[split-analytic]
A map $f=u+iv:\Omega \rightarrow \mathbb{C}^{'}$ is said to be split-analytic, if for any $\eta=s+k^{'}t \in \Omega$ \hspace{0.01cm} there is an open ball, $$B_R(\eta)=\{x+k^{'}y \in \Omega| \sqrt{(s-x)^2+(t-y)^2}<R\}$$
In usual subspace topology of $\Omega$ in $\mathbb{R}^2$ with $f(z)=\Sigma_{0}^{\infty} c_n (z-\eta)^n$, for every $z \in B_R(\eta).$ 
\end{definition}

Unlike holomorphic maps, not all split-holomorphic maps are analytic. One can look at \cite{Fa} page 18 for a counterexample. So, we have to impose this extra condition while stating the identity theorem.
\begin{lem}[Principle of isolated zeros for split-analytic maps]

Suppose $f$ is a split-holomorphic map that is represented by a power series around a zero $\eta$ in the open ball $B_R(\eta)$. Also, assume f is not identically zero on $B_R(\eta)$. Then there is a  $0<r \leq R$ such that $f(z)\neq 0$ whenever $z$ is in $B_r(\eta)-\{ \eta \}$.
\end{lem}

\begin{proof}
There exist coefficients $c_k$ such that $f(z)= \Sigma_{0}^{\infty} c_k (z-\eta)^k$ in $B_R(\eta)$. Let $n \in \mathbb{N}$ be the smallest number such that $c_n \neq 0$. Since $\eta$ is a zero of $f$, $n \geq 1$. Then we have $f(z)=(z-\eta)^n \Sigma_{n+1}^{\infty} c_k (z-\eta)^k= (z-\eta)^n g(z)$. Here $g(\eta) \neq 0$. Now using continuity of $g$ we note that there is a $0<r \leq R $ such that $g(z)\neq 0$ whenever $z$ is in $B_r(\eta)-\{ \eta \}$. Since $f(z)=(z-\eta)^n g(z)$, this concludes the proof.
\end{proof}

Using this we prove the identity theorem. In the following theorem, we are using the definition of accumulation points following the usual topology of $\mathbb{R}^2.$

\begin{theorem}[Identity theorem for split-analytic maps]
Let $\Omega$ $\subseteq$ $\mathbb{C^{'}}$ be a open and connected domain. Also, let $f$ and $g$ be two split-holomorphic maps that are analytic on $\Omega$. If the set $E=\{z \in \Omega|   f(z)=g(z)\}$ contains an accumulation point then $f=g$ on $\Omega$.
\end{theorem}

\begin{proof}
Consider $h=f-g$. Let $\eta \in E$ be an accumulation point. This implies that for any $r>0$, $B_r(\eta)$ contains a point $z \neq \eta$ with $h(z)=0$. Then by lemma 3.1, there is an open ball $B_R(\eta) \subseteq \Omega$  where $h$ is identically zero. Suppose $a\in \Omega \setminus B_R(\eta) $. Since $\Omega$ is path connected, there is a path $\gamma$ with $\gamma(0)=\eta$,$\gamma(1)=a$. Let $t_0=sup\{t \in [0,1] | h(\gamma(s))=0 \forall s \in [0,t]\}$. Note that such a $t_0$ exists, as this set is non-empty and bounded. Due to continuity of $h$, we have $h(\gamma(t_0))=0.$ This imply $\gamma(t_0)$ is a non-isolated zero of $h$. By lemma 3.1, this implies $h$ must be identically zero in a neighborhood of $\gamma(t_0)$. So unless $t_0=1,$ we can always find a $\delta>0$ such that $h(\gamma(t_0+s))=0$ for any $0 <s \leq \delta .$ This contradicts the definition of $t_0$. Thus $t_0$ must be 1, implying $h(a)=0$.
\end{proof}

Now we define the hyperbolic annulus.

\begin{definition}
A hyperbolic annulus is a region of the form $D=\{x+k^{'}y| x>0, a<x^2-y^2<b\}$ in the split-complex plane where a and b are two real numbers.
\end{definition}

We note that, unlike circular annulus, the hyperbolic annulus is simply connected. Now we define split-harmonic conjugate maps.

\begin{definition}
For a split-harmonic map $u: \Omega \subseteq \mathbb{C}^{'} \rightarrow \mathbb{R}$, a map  $v: \Omega \subseteq \mathbb{C}^{'} \rightarrow \mathbb{R}$ is said to be a split-harmonic conjugate of $u$ if $v$ satisfies,
$$u_x=v_y$$  $$u_y=v_x.$$
\end{definition}

We prove that on a simply connected domain, any split-harmonic map has a split-harmonic conjugate.
\begin{theorem}
Any split-harmonic map $u$ on a simply connected domain $\Omega$ has a split-harmonic conjugate $v$.
\end{theorem}
\begin{proof}
Let $v(z)=\int_{z_0}^{z} u_y dx+u_x dy$. We note that by Green's theorem, on any closed curve $C$ in the domain,                        
$$ \int_C u_y dx+ u_x dy = \int _ D (u_{xx}-u_{yy}) dx dy =0$$

Where D is the region bounded by $C$. Thus the map $v$ is well defined and it satisfies $$u_x=v_y$$ $$u_y=v_x.$$ Thus $v$ is a split-harmonic conjugate of $u$ up to a constant.

\end{proof}

We now state the general form of a split-harmonic map on a simply connected domain.

\begin{theorem}[Representation formula for split-harmonic functions]
Let $\Omega$ be a simply-connected domain and $F$: $\Omega \subset $ $\mathbb{C}^{'}$ $\rightarrow$ $\mathbb{C}^{'}$ be a split-harmonic map. Then $F$ can be written as $$F=h+\overline{g}$$ with $h$,$g$ being split-holomorphic maps. This representation formula is unique up to an additive constant.
\end{theorem}

\begin{proof}

 $F$ being split-harmonic implies, $$\frac{\partial F}{\partial \overline{z} \partial z }=0.$$ In particular  $F_z$ is split-holomorphic. We fix a point $z_0$ in $\Omega$. Let $h(z)=\int _{z_0} ^{z} F_z dz.$ is a well-defined map which is split-holomorphic.We refer to  \cite{Rpr} page 16 for proof of well-defineness of $h$. Let $g= \overline{F-h}$. We note that $$g_{\overline{z}}= \frac{\partial \overline{(F-h)}}{\partial \overline{z}}=\overline{F_z-h_z}=0$$  in $\Omega$.

Thus $g$ is also split-holomorphic. We also have $F=h+\overline{g}$. To show uniqueness up to additive constant, suppose $$F=h_1+\overline{g_1}=h_2+\overline{g_2}$$ then $u=h_1-h_2=\overline{g_1-g_2}$ is both split-holomorphic and anti split-holomorphic(I.e, conjugate of a split-holomorphic map). Using Cauchy-Riemann equations in split-complex analysis, one can show that a map that is both split-holomorphic and anti split-holomorphic must be a constant.
\end{proof}
 The above proof is inspired by a similar result for harmonic maps in  \cite{Pd}, page 7.
\section{The singular Bj\"orling problem}
In this section, we give a new proof the singular Bj\"orling problem for timelike minimal surfaces.

Let $\mathbb{H}^1$ denotes the set $\{x+k^{'}y \in \mathbb{C}^{'}|x>0,x^2-y^2=1\}.$

We  recall that for split-exponential map, $e^{k'\theta}=\cosh{\theta}+k^{'}\sinh{\theta}$. Thus any split-complex number $x+k^{'}y$ with $x>0$ and $-x <y<x$, can be written as $z=\rho e^{k' \theta}$ for some real numbers  $\theta$ and $\rho$. Also any point in $\mathbb{H}^1$ can be written as $e^{k^{'}\theta}$ for some real number $\theta$. We refer to \cite{Fa}, page 11 for more details.

We start with a definition.

\begin{definition}
A map $\alpha: \mathbb{H}^1 \rightarrow \mathbb{C^{'}}$ is said to be analytic, if for any point $\eta$ in  $\mathbb{H}^1$, there is a neighborhood of it following the usual topology of $\mathbb{R}^2$, where it can be represented by a powerseries, $\alpha(z)=\Sigma_0 ^{\infty} c_n (z-\eta)^n$.

A curve $\gamma: \mathbb{H}^1$ $\rightarrow$ $\mathbb{L}^3$ is said to be analytic if each of its components is analytic.
\end{definition}

Given analytic $\gamma: \mathbb{H}^1 \rightarrow \mathbb{L}^3$ and $L: \mathbb{H}^1 \rightarrow \mathbb{L}^3$, we define  maps $g_1$ and $g_2$ on $\mathbb{H}^1$ as follows. Let $$g_1(e^{k^{'}\theta})=\left((L_1+k^{'}L_2)+k^{'}(\gamma_1+k^{'}\gamma_2)\right)e^{k^{'}\theta}$$  and  $$g_2(e^{k^{'}\theta})=\left((L_1-k^{'}L_2)+k^{'}(\gamma_1-k^{'}\gamma_2)\right)e^{k^{'}\theta}.$$

Now we state the main result of this section.
 
 \begin{theorem}
Suppose an analytic lightlike curve $\gamma: \mathbb{H}^1 \rightarrow \mathbb{L}^3,$  and an analytic lightlike vector field $L: \mathbb{H}^1 \rightarrow \mathbb{L}^3$ are given with the properties that,
\begin{itemize}
  \item $\langle \gamma^{'}, L \rangle =0$ 
  \item analytic extension  $g_1(z)$ of $g_1(e^{k^{'}\theta})$ and $g_2(z)$ of $g_2(e^{k^{'}\theta})$ satisfy $|g_1(z)|\not \equiv |g_2(z)|$.
\end{itemize}
Then there is a generalized timelike minimal surface $F= (h,\omega)$ defined on some hyperbolic annulus $A(r,R)=\{x+k^{'}y \in \mathbb{C^{'}}| x>0, 0<r<x^2-y^2<R\}$ , $r<1<R$  with singular set  atleast $\mathbb{H}^1$ such that,
$$F(e^{k^{'}\theta})=\gamma(e^{k^{'}\theta}),$$
$$\frac{\partial F}{\partial \rho}|_{e^{k^{'}\theta}}=L(e^{k^{'}\theta})$$

\end{theorem}

\begin{proof}

Since $\gamma$ and $L$ are analytic on $\mathbb{H}^1$, there is a hyperbolic annulus $A(r, R)$ containing  $\mathbb{H}^1$, where analytic extensions of both $\gamma$ and $L$ exists. We construct a harmonic map $h$ on $A(r, R)$ with,
$$h_\theta=\gamma_1+k^{'}\gamma_2,$$
$$h_\rho=L_1+k^{'}L_2.$$

Using, $$h_z=\frac{1}{2}(h_\rho+k^{'}h_\theta)e^{k^{'}\theta}$$ and
$${h}_{\overline{z}}=\frac{1}{2}(h_\rho-k^{'}h_\theta)e^{-k^{'}\theta},$$

We get $$h_z=\frac{1}{2}(L_1+k^{'}L_2+k^{'}(\gamma_1+k^{'}\gamma_2))e^{k^{'}\theta},$$ 
$${h}_{\overline{z}}=\frac{1}{2}(L_1+k^{'}L_2-k^{'}(\gamma_1+k^{'}\gamma_2))e^{-k^{'}\theta}.$$

Using $dh=h_z dz+ {h}_{\overline{z}} d \overline{z}$ from \cite{Fa}, We fix a point $z_0 \in \mathbb{H}^{1}$ and define,
$$h(z)= \int _{z_0} ^{z} dh= \int_{z_0}^{z} h_z dz+ {h}_{\overline{z}} d \overline{z}. $$ I.e, $$ h(z)  \int_{z_0}^{z} (\frac{1}{2}(L_1+k^{'}L_2+k^{'}(\gamma_1+k^{'}\gamma_2))e^{k^{'}\theta}) dz+ (\frac{1}{2}(L_1+k^{'}L_2-k^{'}(\gamma_1+k^{'}\gamma_2))e^{-k^{'}\theta}) d \overline{z}.$$
Where this integral is taken along any path in A(r, R) joining $z_0$ to $z$. We refer to \cite{Fa}, page 14 for similar results. By Stoke's theorem, this map is well defined and satisfies $$h_\theta=\gamma_1+k^{'}\gamma_2,$$
$$h_\rho=L_1+k^{'}L_2.$$

Since $h_z=\frac{1}{2}(L_1+k^{'}L_2+k^{'}(\gamma_1+k^{'}\gamma_2))e^{k^{'}\theta}$ is split-analytic, $h_{z \overline{z}}=0$ and $h$ is split-harmonic. Proof of existence of $\omega$ with $\omega_\theta=\gamma^{'}_3$ and $\omega_\rho=L_3$ is similar.

Now to show this $(h,\omega)$ satisfies $h_z \overline{h_{\overline{z}}}-\omega_z^{2}$ we note by computations that,

$$h_z \overline{h_{\overline{z}}}(e^{k^{'}\theta})=\frac{L_3^{2}+\gamma_3^{'^2}+2k^{'}L_3\gamma_3^{'} e^{2k^{'}\theta}}{4}=\omega_z^{2}(e^{k^{'}\theta}).$$

Thus we note that the split-holomorphic map $h_z \overline{h_{\overline{z}}}-\omega_z^{2}$ is zero on $\mathbb{H}^1$. By using theorem 3.2(Identity theorem) we conclude that it is zero on the entire hyperbolic annulus $A(r, R).$

To show $|h_z|$ not identically same as $|h_{\overline{z}}|$, we note that the maps $g_1(e^{k^{'}\theta})$, $g_2(e^{k^{'}\theta})$ agree with the maps $h_z$ ,$\overline{h_{\overline{z}}}$ on $\mathbb{H}^1$. One can show this by calculations. Now by theorem 3.2(identity theorem), thus the maps $g_1(z),g_2(z)$ are same as $h_z$ ,$\overline{h_{\overline{z}}}$. Now the assumption $|g_1(z)|$ is not identically the same as $|g_2(z)|$ ensures the desired result.

To show singular set contains at least $\mathbb{H}^{1}$, one can compute and prove $(|h_z|-|h_{\overline{z}}|)^2(e^{k^{'}\theta})=0$. Here $|h_z|^2+|h_{\overline{z}}|^2(e^{k^{'}\theta})=\frac{1}{2}(\gamma_3^{{'}_{2}}-L_3^{2})=2|h_z| |h_{\overline{z}}|(e^{k^{'}\theta})=2|\omega_z|^2(e^{k^{'}\theta}).$
\end{proof}

\section{Interpolating a given split-Fourier curve to a  point }
 In this section, we study the problem of interpolating a given spacelike or timelike split-Fourier curve to a point $p$  by a timelike minimal surface. This is similar to a result in \cite{Rpr}.
We first prove that for any analytic split-Fourier curve, split-Fourier coefficients are unique.
\begin{theorem}
For any analytic split-fourier curve $\gamma(\theta)=(\gamma_1+k^{'}\gamma_2,\gamma_3)$   $=(\Sigma_{-\infty}^{\infty} c_n e^{k^{'}n\theta}, \Sigma_{-\infty}^{\infty} d_n e^{k^{'}n\theta})$ , the coefficients $c_n$ and $d_n$ are unique.
\end{theorem}

\begin{proof}
Using induction, we can prove that
$$e^{k^{'}n\theta}=\cosh n \theta+ k^{'} \sinh n \theta.$$

Thus to show the series $\Sigma_{-\infty}^{\infty} c_n e^{k^{'}n \theta}$ has unique coefficients, it is enough to show its split-real and split-imaginary parts $\Sigma_{0}^{\infty} a_n \cosh n \theta$ and $\Sigma_{0}^{\infty} b_n\sinh n \theta$  with $a_n$ and $b_n \in \mathbb{R}$ have unique coefficients. We show that for any function $f$ with a converging series expansion $f(\theta)=\Sigma_{0}^{\infty} a_n \cosh n \theta$, the coefficients $a_n$ are unique.Proof for $\Sigma_{0}^{\infty} b_n\sinh n \theta$ is similar.

 $f$ is real analytic and it has an analytic extension $f(z)=\Sigma_{0}^{\infty} a_n \cosh n z$  within a radius of convergence $R$. Let $i$ denote the complex number with $i^2=-1.$ The point $ib$ is inside this disc for any b with $0\leq b <R$ and we have  $f(ib)=\Sigma_{0}^{\infty} a_n \cosh n ib =\Sigma_{0}^{\infty} a_n \cos n b$. Let $g(b)=f(ib)=\Sigma_{0}^{\infty} a_n \cos nb .$ Then this is a Fourier expansion of $g$ in $(-R,R)$. By multiplying $g$ with a constant if necessary, we assume that $R>\pi.$  Thus the coefficients $a_n$ are Fourier coefficients of $g$ in $[-\pi,\pi]$ and thus they are unique.

\end{proof}

We solve the problem of interpolating a given spacelike or timelike analytic split-Fourier curve to the point $p=(0,0,0)$ using a timelike minimal surface. This is in parallel to a similar result in \cite{Rpr}.

\begin{theorem}
A given  spacelike or timelike  analytic split-Fourier curve  $\gamma(\theta)=(\Sigma_{-\infty}^{\infty} c_n e^{k^{'}n\theta}, \Sigma_{-\infty}^{\infty} d_n e^{k^{'}n\theta})$ can be interpolated as $X(re^{k^{'}\theta})$ to the point $p=(0,0,0)$ with $X(e^{k^{'}\theta})=p$  using a timelike minimal surface $X$, if there is an $r>1$ such that $c_n$ and $d_n$ satisfies,
$$\Sigma _{n=-\infty}^{\infty} 4n(n-m) \frac{c_n\overline{c}_{n-m} r^{2n-m}}{(r^n-1)(r^{n-m}-1)} -\Sigma_{\{(i,j)|i+j=m\}} \hspace{0.1cm} 4ij \frac{d_id_j r^k}{(r^i-1)(r^j-1)}$$ whenever $m\neq 0, n \in \mathbb{Z}$ and

$$\Sigma _{n=-\infty}^{\infty} 4n^2 \left(\frac{\overline{c_n}c_n r^{2n}}{(r^n-1)^2}+\frac{d_n d_{-n}}{(r^n-1)(r^{-n}-1)}\right)=0$$

\end{theorem}

\begin{proof}

Since $\gamma$ is analytic, there is a hyperbolic annulus $A(r, R)$ containing $\mathbb{H}^1$ where its analytic extension exists. We construct a timelike minimal surface $X$ with $X(e^{k^{'}\theta})=p$ and $X(re^{k^{'}\theta})=\gamma(\theta)$ on this hyperbolic annulus, $A(r,R)$.

We consider split-harmonic maps of the form $$h(z)=\Sigma_{-\infty}^{\infty} a_n z^n +b_n \frac{1}{\overline{z}^n}.$$

Using $h(e^{k^{'}\theta})=(0,0)$ we get $a_n=-b_n.$ similary using $h(re^{k^{'}\theta})=\Sigma_0^{\infty} c_n e^{k^{'}n\theta}$, one can compute $a_n$. From here we get, $a_n=\frac{c_n r^n}{(r^{2n}-1)}.$ 

Assuming $$\omega(z)=\Sigma_{-\infty}^{\infty} f_n z^n +g_n \frac{1}{\overline{z}^n}$$ and doing similar computations, we conclude that $f_n=-g_n$ and $f_n=\frac{d_n r^n}{(r^{2n}-1)}.$

Thus the surface $X(z)=(h(z),\omega(z))$ passes through both $p$ and the given curve $\gamma$. To show this is timelike minimal, we compute and get $$h_z \overline{h_{\overline{z}}}-\omega_z^2(e^{k^{'}\theta})=\frac{1}{4}\left(\Sigma_{-\infty}^{\infty}2n a_n e^{k^{'}n\theta}\Sigma_{-\infty}^{\infty}2n \overline{a}_n e^{-k^{'}n\theta}-\left(\Sigma_{-\infty}^{\infty}2nf_n e^{k^{'}n\theta}\right)^2\right).$$
 
From here, comparing coefficients of $e^{k^{'}(n-m)\theta}$ when $n \neq m$ and when $n=m$, one gets the conditions given in the theorem. Thus the above quantity vanishing is equivalent to the conditions given in the result. This being zero implies $h_z \overline{h_{\overline{z}}}-\omega_z^2=0$ on the entire hyperbolic annulus by identity theorem of split-analytic maps. Also $|h_z|$ is not identically same as $|h_{\overline{z}}|$ as $X(re^{k^{'}\theta})$ is a spacelike or timelike curve. Thus $X$ is a timelike minimal surface interpolating both $\gamma$ and $p.$ 
\end{proof}
\section{Interpolating a split-Fourier curve to another split-Fourier curve}
We solve the problem of interpolating a given analytic spacelike or timelike split-Fourier curve to another specified analytic split-Fourier curve.

\begin{theorem}
Given a  spacelike or timelike analytic  split-Fourier curve  $\gamma(\theta)=(\Sigma_{-\infty}^{\infty} c_n e^{k^{'}n\theta}, \Sigma_{-\infty}^{\infty} d_n e^{k^{'}n\theta})$ and a analytic split-Fourier curve $\alpha(\theta)=(\Sigma_{-\infty}^{\infty} l_n e^{k^{'}n\theta}, \Sigma_{-\infty}^{\infty} m_n e^{k^{'}n\theta})$, for $r>1$ let $$a_n(r)=\frac{r^{n}c_n-l_n}{r^{2n}-1}$$ and $$f_n(r)=\frac{r^{n}d_n-m_n}{r^{2n}-1}.$$$\gamma$  can be interpolated as $X(re^{k^{'}\theta})$ to $\alpha$ as $X(e^{k^{'}\theta})$ by a timelike minimal surface $X$  if there is an $0<r<1$ such that $a_n(r)$ and $f_n(r)$ satisfies,

$$\Sigma_{n=-\infty}^{\infty} 4n(n-m)a_n(r)(\overline{a_{n-m}(r)}-\overline{l}_{n-m})-\Sigma_{\{(i,j)| i+j=m\}} \hspace{0.1cm} 4ij f_i(r) f_j(r)=0,$$

for any $m \neq 0$ \hspace{0.01cm} and,
$$\Sigma_{n=-\infty}^{\infty}4n^2\left(a_n(r)(\overline{a_n(r)}-\overline{l}_n)+f_n(r)f_{-n}(r)\right)=0.$$
\end{theorem}

\begin{proof}
The proof is similar to that of theorem 5.2. We consider split-harmonic maps of the form,
$$h(z)=\Sigma_{-\infty}^{\infty} a_n z^n +b_n \frac{1}{\overline{z}^n}$$ and
$$\omega(z)=\Sigma_{-\infty}^{\infty} f_n z^n +g_n \frac{1}{\overline{z}^n}.$$
$X(e^{k^{\theta}})=\alpha$ imply $b_n=l_n-a_n$ and $g_n=m_n-f_n$. Similarly $X(re^{k^{'}\theta})=\gamma$ implies $$a_n=\frac{r^{n}c_n-l_n}{r^{2n}-1}$$ and $$f_n=\frac{r^{n}d_n-m_n}{r^{2n}-1}$$.

Thus with these choices of $a_n, b_n, f_n, g_n$ the surface $X$ passes through both $\alpha$ and $\gamma$. To show this is timelike minimal we compute and get,
$$h_z \overline{h_{\overline{z}}}-\omega_z^2(e^{k^{'}\theta})=\frac{1}{4}\left(\Sigma_{-\infty}^{\infty}(2n a_ne^{k^{'}n\theta})\Sigma_{-\infty}^{\infty}( 2n(\overline{a}_n-\overline{l}_n) e^{-k^{'}n\theta}-\left(\Sigma_{-\infty}^{\infty}2nf_n e^{k^{'}n\theta}\right)^2\right).$$
 
From here, comparing coefficients of $e^{k^{'}(n-m)\theta}$ when $n \neq m$ and when $n=m$, one gets the conditions given in the theorem. The above quantity vanishing is equivalent to the conditions given in the result. This being zero implies $h_z \overline{h_{\overline{z}}}-\omega_z^2=0$ on the entire hyperbolic annulus by identity theorem of split-analytic maps. Also $|h_z|$ is not identically same as $|h_{\overline{z}}|$ as $X(re^{k^{'}\theta})$ is a spacelike or timelike curve. Thus $X$ is a timelike minimal surface interpolating both $\gamma$ and $\alpha.$ 
\end{proof}
\section*{Acknowledgement}
I would like to thank my MS thesis supervisor Prof. Rukmini Dey(ICTS Bangalore) for suggesting the problem to me and for all the interesting and helpful discussions we had during the project. I want to express my gratitude to ICTS Bangalore for their hospitality.  This work was done as a part of my master's thesis in IISER Pune.

\end{document}